\documentclass[12pt]{amsart}
\usepackage{epsfig,color}

\usepackage{mathrsfs}

\theoremstyle{definition}

\newtheorem{theorem}{Theorem}[section]
\newtheorem{definition}[theorem]{Definition}

\newtheorem{proposition}[theorem]{Proposition}
\newtheorem{lemma}[theorem]{Lemma}
\newtheorem{remark}[theorem]{Remark}
\newtheorem{corollary}[theorem]{Corollary}

\numberwithin{equation}{section}

\renewcommand\div{\operatorname{div}}

\newcommand{\wt}{\widetilde}
\newcommand{\pr}{\partial}

\newcommand{\ol}{\overline}
\newcommand{\Lap}{\Delta}

\renewcommand*\d{\mathop{}\!\mathrm{d}}

\title[Moving-centre monotonicity formulae]{Moving-centre monotonicity formulae for minimal submanifolds and related equations}

\author{Jonathan J. Zhu}%
\address{Department of Mathematics,
Harvard University, Cambridge, MA 02138, USA}
\email{jjzhu@math.harvard.edu}

\begin{document}
\date{\today}
\maketitle

\begin{abstract}
Monotonicity formulae play a crucial role for many geometric PDEs, especially for their regularity theories. For minimal submanifolds in a Euclidean ball, the classical monotonicity formula implies that if such a submanifold passes through the centre of the ball, then its area is at least that of the equatorial disk. Recently Brendle and Hung proved a sharp area bound for minimal submanifolds when the prescribed point is not the centre of the ball, which resolved a conjecture of Alexander, Hoffman and Osserman. Their proof involves asymptotic analysis of an ingeniously chosen vector field, and the divergence theorem.

In this article we prove a sharp `moving-centre' monotonicity formula for minimal submanifolds, which implies the aforementioned area bound. We also describe similar moving-centre monotonicity formulae for stationary $p$-harmonic maps, mean curvature flow and the harmonic map heat flow. 
\end{abstract}

\section{Introduction}

For many geometric partial differential equations, monotonicity formulae play an essential role and their discovery often leads to deep and fundamental results for those systems. Monotonicity is a particularly useful tool in the study of variational problems, and for regularity theory (see for example \cite{almgren, bartnik, CMmono, evans, EWW, simon, white} and references therein). These formulae often control the evolution of energy-type quantities with respect to changes in scale, or time.

An important example is the classical monotonicity formula for minimal submanifolds - critical points of the area functional - which states:

\begin{proposition}
Let $\Sigma^k$ be a minimal submanifold in $\mathbb{R}^n$. Then so long as $\pr\Sigma \cap \overline{B_r^n} = \emptyset$, we have 
\begin{equation}
\label{eq:intromin}
\frac{d}{dr} \left(r^{-k} |\Sigma \cap B_r^n |\right) = r^{-k-1} \int_{\Sigma\cap \pr B^n_r} \frac{|x^\perp|^2}{|x^T|}  \geq 0.
\end{equation}
\end{proposition}

Here $B_r^n=B^n(0,r)$ denotes the Euclidean ball of radius $r$ about the origin in $\mathbb{R}^n$. Thus the area ratio $r^{-k}|\Sigma\cap B_r^n|$ is monotone on balls with fixed centre, and so comparing to the limiting density as $r\searrow 0$ yields that any minimal submanifold $\Sigma^k\subset B_r^n$ with $\pr\Sigma \subset \pr B_r^n$, which passes through the origin, satisfies the sharp area bound 
\begin{equation}
\frac{|\Sigma \cap B_r^n|}{r^k} \geq |B_1^k| , 
\end{equation}
with equality if and only if $\Sigma$ is a flat $k$-disk. 

In the case that the minimal submanifold $\Sigma^k \subset B_r^n$ does not necessarily pass through the centre of the ball, Alexander, Hoffman and Osserman \cite{AHO} conjectured (see also \cite{osserman}) the following sharp area bound, which has recently been proven in full generality by Brendle and Hung \cite{brendlehung} (see also Corollary \ref{thm:brendlehung}). %
Alexander and Osserman had previously proven the conjecture only in the case of simply connected surfaces \cite{AO}. 

\begin{theorem}[\cite{brendlehung}]
\label{thm:introbrendlehung}
Let $\Sigma^k$ be a minimal submanifold in the ball $B_r^n $ with $\pr \Sigma \subset \pr B_r^n$. Then \begin{equation}
\label{eq:introbh}
\frac{|\Sigma\cap B_r^n|}{(r^2-d^2)^\frac{k}{2}} \geq |B^k_1|,\end{equation} where $d=d(0,\Sigma)$ is the distance from $\Sigma$ to the centre of the ball. 
\end{theorem}

The proof of Theorem \ref{thm:introbrendlehung} by Brendle-Hung involves the choice of a clever, but somewhat geometrically mysterious, vector field $W$. They apply the divergence theorem to $W$ away from small balls $B_\epsilon(y)$, where $y\in \Sigma \cap B_r$, and obtain the estimate in the limit as $\epsilon\rightarrow 0$. 

In this paper, we show that the area bound (\ref{eq:introbh}) in fact arises from a sharp `moving-centre' monotonicity formula, in which the centres of the extrinsic balls are allowed to move, and the scale is adjusted in a particular manner:
 
\begin{definition}
Fix $y\in B_R^n$. For $s\geq 0$ we let \begin{equation}E_s = B^n((1-s)y, r(s))\end{equation} denote the ball with centre $(1-s)y$ and radius $r(s) := \sqrt{s(R^2-|y|^2) + s^2|y|^2}.$
\end{definition}

Our main theorem is then as follows (see also Theorem \ref{thm:ocmmin}):

\begin{theorem}
\label{thm:intromain}
Let $\Sigma^k$ be a minimal submanifold in $\mathbb{R}^n$ and $y \in B_R^n, E_s, r(s)$ be as above. Then so long as $\pr \Sigma \cap \overline{E_s} =\emptyset$, we have 
\begin{equation}
\frac{d}{ds}\left( \frac{|\Sigma \cap E_s|}{(r(s)^2 - d(s)^2)^\frac{k}{2}}\right) =\frac{s^{-\frac{k+2}{2}}}{2(R^2-|y|^2)^\frac{k}{2}} \int_{\Sigma \cap \pr E_s} \frac{|(x-y)^\perp|^2+s^2|y^T|^2}{|(x-y+sy)^T|} \geq 0 , 
\end{equation}
where $d(s)=s|y|$ is the distance from $y$ to the centre of the ball $E_s$. 
\end{theorem}

As $s\searrow 0$, the monotone quantity \begin{equation}(r(s)-d(s))^{-\frac{k}{2}}|\Sigma\cap E_s|\end{equation} again picks up the density at $y$. On the other hand, $r(1)=R$, so taking $y\in\Sigma\cap B_1 = \Sigma \cap E_1$ to be the point with $|y|=d(0,\Sigma)$ and comparing $s=0$ to $s=1$ indeed yields the area bound (\ref{eq:introbh}). 

This new proof using our monotonicity formula thus offers a new perspective on Theorem \ref{thm:introbrendlehung}. In fact, the choice of radius $r(s)$ above is constrained at a technical level in the proof of the monotonicity Theorem \ref{thm:intromain} (see Lemma \ref{lem:mindf}); with this choice in hand one may then observe that Brendle-Hung's vector field $W$ takes a relatively simple form in terms of a function $f$ whose sub-level sets are the balls $E_s$. Interestingly, this function also seems to emerge organically in the study of the exponential transform \cite{tkachev}. %

Moreover, the distinction between having an area bound compared to having monotonicity of an area ratio is significant (although potentially subtle). For example, for minimal submanifolds $\Sigma^k \subset N^n$, where $N$ has sectional curvature bounded above by $K$, it holds that \begin{equation}\frac{|\Sigma^k \cap \bar{B}^{N}_r|}{A_0(r)} \geq 1,\end{equation} where $\bar{B}^N_r$ is the geodesic ball in $N$ of radius $r$, and $A_0(r)$ is the area of the geodesic ball of radius $r$ in the $k$-dimensional space form of curvature $K$. (If $K>0$ then $r$ must be less than $\frac{\pi}{2\sqrt{K}}$.) For $K\leq 0$, the quantity on the left is in fact monotone non-decreasing \cite{anderson}, but for $K>0$, such a monotonicity for the area ratio is not known. Instead, one can prove the area bound by proving a monotonicity for the quantity $A_0(r)^{-1} \int_{\bar{B}_r^N} |\nabla \rho|^2$, where $\rho$ is the distance function on $N$ \cite{gulliverscott}. 

In \cite{AHO} it was shown that Theorem \ref{thm:introbrendlehung} would be a consequence of the sharp isoperimetric inequality $|\pr \Sigma|^k \geq k^k |B_1^k| |\Sigma|^{k-1}$ for minimal submanifolds $\Sigma^k$, although the latter is not presently known in general. In fact, isoperimetric inequalities are closely related to monotonicity properties of minimal submanifolds, particularly in the work of Choe (see \cite{choe05} for a survey). One might therefore hope that the approaches to Theorems \ref{thm:introbrendlehung} and \ref{thm:intromain} could provide some insight towards the sharp isoperimetric inequality. 

 The proof of Theorem \ref{thm:intromain} may be found in Section \ref{sec:min}. It uses only the coarea formula and the divergence theorem, except that the key is to apply the divergence theorem to a different vector field at each $s$. Thus, as for the classical monotonicity formula, our moving-centre monotonicity formula also holds for stationary varifolds, and also admits an almost-monotonicity for submanifolds with $L^p$-bounded mean curvature. We also provide a second proof using only the divergence theorem, which may clarify Brendle-Hung's vector field $W$. 
 
\subsection{Monotonicity formulae for related geometric systems}
In Theorem \ref{thm:ocmmcf}, Theorem \ref{thm:ocmharm1} and Theorem \ref{thm:ocmhmheat} we present moving-centre monotonicity formulae for the mean curvature flow, $p$-harmonic maps and harmonic map heat flow respectively. 

In Section \ref{sec:harm} we consider stationary $p$-harmonic maps $(p>1)$; for the fixed-centre monotonicity in this setting one may consult \cite{naber}. The minimal submanifold case is morally the $p=1$ case; in fact it is the critical case for our moving-centre monotonicity in the sense that for $p>1$, there is a term of the wrong sign that cannot be fully absorbed in the naive manner. The offending term may be handled either by accepting an almost-monotonicity and multiplying the monotone quantity by a correcting factor, or instead by adjusting the scale more carefully. For $p>1$ the latter method applies only if the centre does not move too quickly. Interpolating between the two methods is what gives rise to our family of monotonicity formulae for these elliptic systems.

In Sections \ref{sec:mcf} and \ref{sec:hmheat} we present our results for certain parabolic systems, namely the mean curvature flow and harmonic map heat flow respectively. The respective fixed-centre monotonicity formulae are due to Huisken \cite{huisken} and Struwe \cite{struwe}. In both cases, the monotone quantity involves a global energy-type integral against a Gaussian weight. Thus, unlike the elliptic case, an additional factor to compensate for the motion of the centre appears to be unavoidable. However, for these geometric flows, we obtain a monotonicity for motion of the centre along any $C^1$ path, not just on straight lines. 

Finally, one should note that a type of moving-centre monotonicity was used by Colding and Minicozzi \cite{CMgeneric} to show that the entropy of a mean curvature flow self-shrinker is achieved by the Gaussian area at the natural centre and scale of the self-shrinker. This is an important step in their classification of entropy-stable, or generic, self-shrinkers (see also \cite{zhu}). For the reader's convenience we briefly describe their result in Section \ref{sec:shrinker}.

\subsection{Notation}

In Euclidean space we will always use $x$ to denote the position vector. 

When working with submanifolds, we will use $\ol{\nabla}$ to denote the ambient connection and $\nabla$ for the induced connection on a submanifold, with $D$ reserved for the Euclidean connection. We use $y^T$ for the projection of a vector $y$ to the tangent bundle, and $y^\perp$ for the projection to the normal bundle. 

When dealing with maps between manifolds $M\rightarrow N \hookrightarrow \mathbb{R}^n$, we will unambiguously use $\nabla$ for the connection on $M$. If $M=\mathbb{R}^m$ we use lower, Latin indices for coordinates on $\mathbb{R}^m$ and upper, Greek indices for coordinates on $\mathbb{R}^n$. Repeated indices are summed throughout, unless otherwise noted, and commas denote derivatives. In this setting we use $\cdot$ to distinguish contraction on $M$ from full contraction $\langle ,\rangle$. 

We use $I$ to denote an open interval in $\mathbb{R}$. 

We will need the coarea formula, which states that for a proper Lipschitz function $f$ and a locally integrable function $u$ on a manifold $M$, one has 
\begin{equation}
\int_{\{f\leq t\}} u|\nabla f| = \int_{-\infty}^t \d\tau \int_{\{f=\tau\}} u.
\end{equation}
(See for instance \cite{CMbook}, or \cite{simon} for more general statements including for varifolds.)

We denote by $B^k(p,r)$ the (open) Euclidean ball in $\mathbb{R}^k$ with centre $p$ and radius $r$. For simplicity we will write $B^k_r = B^k(0,r)$. We will often omit the dimension when it is clear from context.

We typically prefer to derive our monotonicity formulae in differential form; derivatives with respect to the time or scale parameter should be interpreted in the distribution sense.

Additional notation and background specific to each setting will be explained in the respective sections of this paper.

\subsection*{Acknowledgements}

We are indebted to Nick Edelen, as well as Otis Chodosh, for bringing the initial problem to the author's attention and for stimulating discussions. The author would also like to thank Prof. Bill Minicozzi for numerous valuable suggestions.

\section{Minimal submanifolds}
\label{sec:min}

Recall that the divergence theorem or first variation formula for submanifolds states that \begin{equation}\label{eq:mcfdivthm}\int_{\Sigma} \div_{\Sigma} X = - \int_{\Sigma} \langle X,\vec{H}\rangle + \int_{\pr\Sigma} \langle X,\nu\rangle \end{equation} for any smooth compactly supported ambient vector field $X$, and that minimal submanifolds are those that satisfy $\vec{H}=0$. Here $\nu$ is the outward unit normal of $\pr \Sigma$ with respect to $\Sigma$. 

The classical monotonicity formula for minimal submanifolds (see for instance \cite{simon}, or \cite{CMbook}) states that:

\begin{proposition}
Let $\Sigma^k$ be a minimal submanifold in the ball $B_{\bar{r}}^n \subset \mathbb{R}^n$ with $\pr \Sigma \subset \pr B_{\bar{r}}^n$. Then for $0<r<\bar{r}$ one has 
\begin{equation}
\frac{d}{dr} \left(r^{-k} |\Sigma \cap B_r|\right) = r^{-k-1} \int_{\Sigma \cap \pr B_r} \frac{|x^\perp|^2}{|x^T|} .
\end{equation}

Equivalently, for $0<r<t<\bar{r}$, we have 
\begin{equation}
t^{-k}|\Sigma\cap B_t| - r^{-k}|\Sigma\cap B_r| = \int_{\Sigma \cap B_t\setminus B_r} \frac{|x^\perp|^2}{|x|^{k+2}}
\end{equation}

In particular, the area ratio $r^{-k} |\Sigma \cap B_r|$ is non-decreasing in $r$, and is constant if and only if $\Sigma$ is a cone (with vertex at 0).  
\end{proposition}

In order to state our moving-centre monotonicity formula, we first define a family of extrinsic balls on which to view the submanifold. Note that in this section, we work with equivalent forms of Theorems \ref{thm:introbrendlehung} and \ref{thm:intromain}, in which the ball $B_R$ is scaled back to the unit ball $B_1$. In particular $y$ will denote a point in the unit ball. 

\begin{definition}
\label{defn:Eballs}
Fix $y\in B_1^n$. For $s\geq 0$, denote the ball \begin{equation}E_s=B^n\left((1-s)y, r(s)\right)\subset \mathbb{R}^n ,\end{equation} where \begin{equation}
\label{eq:choosers} 
r(s)=\sqrt{s(1-|y|^2)+s^2|y|^2)}.\end{equation} Note that the $E_s$, $s\geq 0$ foliate the half-space defined by $\langle x,y\rangle < \frac{1+|y|^2}{2}$. 

These balls may also be realised as sub-level sets, \begin{equation}E_s=\{0\leq f < s\},\end{equation} where explicitly  \begin{equation}f(x) = \frac{|x-y|^2}{1-2\langle x,y\rangle +|y|^2} = \frac{|x-y|^2}{1-|x|^2+|x-y|^2}.\end{equation} 
\end{definition}

The moving-centre monotonicity formula is then as follows:

\begin{theorem}
\label{thm:ocmmin}
Let $\Sigma^k$ be a minimal submanifold in $E_{\bar{s}}\subset \mathbb{R}^n$ with $\pr\Sigma \subset \pr E_{\bar{s}}$ for some $\bar{s}$. Then for $0<s<\bar{s}$, we have that
\begin{equation}
\label{eq:ocmmindiff1}
\frac{d}{ds} \left(s^{-\frac{k}{2}}|\Sigma \cap E_s |\right) = \frac{s^{-\frac{k+2}{2}}}{2} \int_{\Sigma \cap \pr E_s} \frac{|(x-y)^\perp|^2+s^2|y^T|^2}{|(x-y+sy)^T|} .
\end{equation}
Equivalently, for $0<s<t<\bar{s}$, we have 
\begin{equation}
\label{eq:ocmminint}
t^{-\frac{k}{2}}|\Sigma \cap E_t| - s^{-\frac{k}{2}}|\Sigma \cap E_s| = \int_{\Sigma \cap E_t\setminus E_s} f^{-\frac{k}{2}}\left( \frac{|(x-y)^\perp|^2+f^2|y^T|^2}{|x-y|^2}\right).
\end{equation}

In particular, the quantity \begin{equation} s^{-\frac{k}{2}}|\Sigma \cap E_s|\end{equation} is nondecreasing, and is constant if and only if $\Sigma$ is a flat disk orthogonal to $y$. 
\end{theorem}

\begin{remark}
Our proof will only require the coarea and first variation formulae, so it can be seen that Theorem \ref{thm:ocmmin} also holds for stationary varifolds $\Sigma$, except that in the equality case one must allow for cones with vertex at $y$ that are orthogonal to $y$. 

Note that the above is indeed equivalent to Theorem \ref{thm:intromain}, since rearranging (\ref{eq:choosers}) yields that $r(s)^2 - s^2|y|^2 = s(1-|y|^2)$. 

Taking $y=0$ of course recovers the classical monotonicity formula for minimal submanifolds.
\end{remark}

It may be helpful to note that if $\Sigma_0$ is indeed a flat $k$-plane orthogonal to $y$, then any $x\in \Sigma_0$ satisfies $|x|^2 = |y|^2 + |x-y|^2$ and hence $\Sigma_0 \cap E_s$ is a flat $k$-disk of radius $\sqrt{s(1-|y|^2)}$ as expected.   

\begin{corollary}[\cite{brendlehung}]
\label{thm:brendlehung}
Let $\Sigma^k$ be a minimal submanifold in the unit ball $B_1^n \subset \mathbb{R}^n$ with $\pr \Sigma \subset \pr B_1^n$ and $y \in \Sigma$. Then $|\Sigma| \geq |B^k_1|(1-|y|^2)^\frac{k}{2}$, with equality if and only if $\Sigma$ is a flat disk orthogonal to $y$.  
\end{corollary}

Corollary \ref{thm:brendlehung} was first proven in full generality by Brendle and Hung \cite{brendlehung}, using a carefully chosen vector field that we will return to later in this section. 

\begin{proof}[Proof of Corollary \ref{thm:brendlehung} using Theorem \ref{thm:ocmmin}]
As $s \searrow 0$, the balls $E_s$ are asymptotic to the balls $B(y, \sqrt{s(1-|y|^2)})$. So the limit $(1-|y|^2)^{-\frac{k}{2}}\lim_{s\rightarrow 0} s^{-\frac{k}{2}}|\Sigma \cap E_s|$ is equal to the density of $\Sigma$ at $y$, which is at least 1 since $y\in\Sigma$. On the other hand $E_1 = B(0,1)$, and thus comparing $s=0$ to $s=1$ using Theorem \ref{thm:ocmmin} immediately yields Corollary \ref{thm:brendlehung}.  
\end{proof}

We now turn to the proof of Theorem \ref{thm:ocmmin}, for which we first calculate the gradient of $f$:

\begin{lemma}
\label{lem:mindf}
Let $r(s), E_s$ and the function $f$ with $\pr E_s = \{f=s\}$ be as in Definition \ref{defn:Eballs}. Then whereever $f>0$, we have that
\begin{equation}
\frac{Df}{2f}=  \frac{x-y+fy}{|x-y|^2}.
\end{equation}
\end{lemma}
\begin{proof}
One may verify this using the explicit formula for $f$, but it is more illuminating to proceed using only the characterisation by level sets together with the choice of $r(s)$.
 
Indeed, let $\rho(s) = r(s)^2 = s(1-|y|^2) + s^2|y|^2$. By construction, the level sets of $f$ are the spheres $\pr E_s$ with centre $(1-s)y$ and radius $r(s)$. So the function $f$ satisfies
\begin{equation}
\label{eq:centreradiusmin}
|x-(1-f)y|^2 = \rho(f)
\end{equation}
or in somewhat expanded form 
\begin{equation}\label{eq:minDs2}|x-y|^2 = -2f\langle x-y,y\rangle+\rho(f) - f^2|y|^2  .\end{equation} 
Moreover $Df$ must be proportional to $x-(1-f)y$, so implicitly differentiating (\ref{eq:centreradiusmin}), we find that \begin{equation}\label{eq:minDf}\frac{1}{2}D f = \frac{x-y+fy}{\rho'(f) - 2\langle x-y, y\rangle -2f|y|^2 } .\end{equation} On the other hand, we note that $\rho$ satisfies the differential equation \begin{equation} s(\rho'(s) - 2s|y|^2) = \rho(s) - s^2|y|^2  .\end{equation}

Therefore using (\ref{eq:minDs2}) yields 
\begin{equation}
\frac{1}{2}Df =  \frac{f(x-y+fy)}{|x-y|^2}.
\end{equation}
\end{proof}

\begin{proof}[Proof of Theorem \ref{thm:ocmmin}]
The outward unit normal $\nu$ to $ \Sigma\cap \{f=s\}$ considered as the boundary of $\Sigma \cap \{f< s\}$ is given by $\nu = \frac{\nabla f}{|\nabla f|}$. For fixed $s$ we let $X_s$ be a vector field with $\div_\Sigma X_s \equiv k$, to be chosen later. By the divergence theorem we would then have \begin{equation} |\Sigma \cap E_s | = \frac{1}{k} \int_{\Sigma \cap \{f=s\}} \langle X_s, \frac{\nabla f}{|\nabla f|}\rangle.\end{equation}

On the other hand, by the coarea formula we have \begin{equation}|\Sigma\cap E_s| = \int_0^s \d\tau \int_{\Sigma\cap\{f=\tau\}} \frac{1}{|\nabla f|}.\end{equation} 

This allows us to compute the derivative of $|\Sigma \cap E_s|$ as an integral over $\Sigma \cap \pr E_s$, using Lemma \ref{lem:mindf}: \begin{eqnarray} \frac{d}{ds} \left(s^{-\frac{k}{2}}|\Sigma \cap E_s|\right) &=& s^{-\frac{k+2}{2}} \int_{\Sigma \cap\{f=s\}} \frac{1}{|\nabla f|}\left( s-\frac{1}{2}\langle X_s,\nabla f\rangle \right)\\&=&\nonumber s^{-\frac{k+2}{2}} \int_{\Sigma \cap\{f=s\}} \frac{s}{|\nabla f|}\left( 1-\langle X_s,\frac{\nabla f}{2f}\rangle \right) \\&=&\nonumber s^{-\frac{k}{2}} \int_{\Sigma \cap\{f=s\}} \frac{1}{|\nabla f|}\left( 1-\frac{\langle X_s, (x-y+sy)^T \rangle}{|x-y|^2} \right) .\end{eqnarray}

Choosing $X_s = x-y-sy$, we indeed have $\div_\Sigma X_s = k$ since $s$ is fixed, and moreover \begin{equation} \langle X_s, (x-y+sy)^T \rangle = |(x-y)^T|^2 - s^2 |y^T|^2.\end{equation} 

Thus 
\begin{equation}
\label{eq:ocmmindiff0}
\frac{d}{ds} \left(s^{-\frac{k}{2}}|\Sigma \cap E_s |\right) = s^{-\frac{k}{2}} \int_{\Sigma \cap\{f=s\}} \frac{1}{|\nabla f|}\left(\frac{|(x-y)^\perp|^2+s^2|y^T|^2}{|x-y|^2} \right).
\end{equation}

Using Lemma \ref{lem:mindf} to replace $|\nabla f|$ yields the differential form (\ref{eq:ocmmindiff1}), whilst integrating (\ref{eq:ocmmindiff0}) using the coarea formula a second time gives the integral form (\ref{eq:ocmminint}). It is clear from either formulation that $s^{-\frac{k}{2}}|\Sigma \cap E_s|$ is constant if and only if $(x-y)^\perp\equiv 0$ and $y^T\equiv 0$ on $\Sigma$. The first condition implies that $\Sigma$ is a cone with vertex at $y$ (hence a plane, if $\Sigma$ is smooth), and the second implies that $\Sigma$ is orthogonal to $y$. 
\end{proof}

\begin{remark}
As with the classical monotonicity formula, one still obtains an almost-monotonicity if one assumes only an $L^p$ bound for the mean curvature $\vec{H}$, by following the proof above and bounding the vector field $X_s$ on the set $E_s$ to handle the extra term. 

For instance, if the mean curvature is bounded by $|\vec{H}|\leq C_H$, then using the bound \begin{equation} |X_s|\leq 2s|y| + \sqrt{s(1-|y|^2)+s^2|y|^2} \leq 3s|y| + \sqrt{s(1-|y|^2)}\end{equation} on $\{0\leq f\leq s\}$, one still obtains a monotone quantity after multiplying by the integrating factor $\exp(kC_H\mu)$, where \begin{equation}\mu = \frac{1}{2}\int \left(3|y| + \sqrt{\frac{1-|y|^2}{s}}\right) \d s = \frac{3}{2}s|y| + \sqrt{s(1-|y|^2)} . \end{equation}
\end{remark}

For completeness, we now give another proof of Theorem \ref{thm:ocmmin} that is instead motivated by, and utilises, the work of Brendle-Hung \cite{brendlehung}. By using the divergence theorem (twice), this proof recovers the integral formulation (\ref{eq:ocmminint}).

\begin{proof}[Alternative proof of Theorem \ref{thm:ocmmin}]
With $f$ defined as above, the vector field utilised by Brendle and Hung may be written as \begin{equation}W = -\frac{1}{k}(f^{-\frac{k}{2}}-1)(x-y) + F(f)y,\end{equation} \begin{equation}F(t) := \begin{cases} \frac{1}{k-2}(t^{\frac{2-k}{2}}-1) &, k>2 \\ -\frac{1}{2}\log t &, k=2.\end{cases}\end{equation}

Setting $W_0 := \frac{1}{k}(x-y)-W$, the computations of Brendle and Hung \cite{brendlehung} yield that \begin{eqnarray}\div_\Sigma W_0 &=& 1-\div_\Sigma W \\&=&\nonumber \frac{f^{-\frac{k}{2}}|(x-y)^\perp|^2 + f^{-\frac{k-4}{2}} |y^T|^2}{|x-y|^2}  \end{eqnarray} 

On the other hand, for any $0<s<\bar{s}$, when restricted to $\pr E_s=\{f=s\}$ we have $W_0 = \frac{1}{k}s^{-\frac{k}{2}}(x-y) - F(s)y$. So applying the divergence theorem we find that \begin{equation} \int_{\Sigma\cap \{f=s\}} \langle W_0,\nu\rangle = \int_{\Sigma\cap\{f=s\}} \langle \frac{1}{k}s^{-\frac{k}{2}}(x-y) - F(s)y,\nu\rangle = s^{-\frac{k}{2}} | \Sigma \cap \{f\leq s\}|,\end{equation} since $\div_\Sigma x=k$ and $\div_\Sigma y=0$. 

Then applying the divergence theorem a second time, for any $0<s<t < \bar{s}$ we indeed have \begin{eqnarray}\nonumber t^{-\frac{k}{2}}|\Sigma\cap\{f< t\}| - s^{-\frac{k}{2}}|\Sigma\cap\{f< s\}|&=& \int_{\Sigma\cap \{f=t\}} \langle W_0,\nu\rangle - \int_{\Sigma\cap \{f=s\}} \langle W_0,\nu\rangle \\&=&\nonumber \int_{\Sigma\cap \{s< f < t\}} \div_\Sigma W_0 \\&=&\int_{\Sigma \cap\{s < f < t\}} \frac{f^{-\frac{k}{2}}|(x-y)^\perp|^2 + f^{-\frac{k-4}{2}} |y^T|^2}{|x-y|^2} . \end{eqnarray}
\end{proof}

\section{Mean curvature flow}
\label{sec:mcf}

A one-parameter family of submanifolds $\Sigma_t^k \subset \mathbb{R}^n$ flows by mean curvature if it satisfies $\pr_t x = \vec{H}$. For submanifolds moving by mean curvature, Huisken \cite{huisken} (see also \cite{ecker}) discovered a monotonicity for Gaussian areas. For a submanifold $\Sigma^k \subset \mathbb{R}^n$, define the Gaussian density with centre $x_0\in\mathbb{R}^n $ and scale $t_0>0$ by \begin{equation}
\label{eq:gaussianarea}
F_{x_0,t_0}(\Sigma) = \int_\Sigma \rho_{x_0,t_0} , \text{ where } \rho_{x_0,t_0}(x)= (4\pi t_0)^{-\frac{k}{2}} \exp\left(-\frac{|x-x_0|^2}{4t_0}\right) . 
\end{equation}

For a mean curvature flow $\Sigma^k_t$, Huisken's monotone quantity is obtained by decreasing the scale as the flow progresses in time. Specifically, given a spacetime centre $(x_0,t_0)$, the monotone quantity will be $F_{x_0, t_0-t}(\Sigma_t) = \int_{\Sigma_t} \Phi_{x_0,t_0}$, where we have set \begin{equation}\Phi_{x_0,t_0}(x,t)=\rho_{x_0,t_0-t}(x)= (4\pi (t_0-t))^{-\frac{k}{2}} \exp\left(- \frac{|x-x_0|^2}{4(t_0-t)}\right).\end{equation} 

\begin{theorem}[\cite{huisken}]
Suppose that $\{\Sigma^k_t\}_{t\in I}$ is a mean curvature flow in $\mathbb{R}^n$ and fix $x_0\in\mathbb{R}^n$, $t_0\in\mathbb{R}$. Further suppose that $\int_{\Sigma_t} \Phi_{x_0,t_0} <\infty$ for all $t\in I$ with $t<t_0$. 

Then for all such times, one has
\begin{equation}
\frac{d}{dt} \int_{\Sigma_t} \Phi_{x_0,t_0} = -\int_{\Sigma_t} \left| \vec{H} + \frac{(x-x_0)^\perp}{2(t_0-t)} \right|^2 \Phi_{x_0,t_0}.
\end{equation}

In particular, $\int_{\Sigma_t} \Phi_{x_0,t_0}$ is non-increasing for $t<t_0$, and is constant if and only if $\Sigma_t$ is a self-shrinking soliton that shrinks to $(x_0,t_0)$. 
\end{theorem}

Recall that a self-shrinking soliton, which shrinks to $(x_0,t_0)$, is a mean curvature flow that is invariant under (backwards) parabolic dilations about $(x_0,t_0)$, or equivalently satisfies $\vec{H} = - \frac{(x-x_0)^\perp}{2(t_0-t)}$ on each $\Sigma_t$.

Motivated by Theorem \ref{thm:brendlehung}, we give a sharp monotonicity formula for mean curvature flow in which the Gaussian centre $x_0$ is allowed to move in time. Namely, we prove that:

\begin{theorem}
\label{thm:ocmmcf}
Suppose that $\{\Sigma^k_t\}_{t\in I}$ is a mean curvature flow in $\mathbb{R}^n$ and fix $t_0\in\mathbb{R}$. Let $y=y(t)$ be a smooth curve in $\mathbb{R}^n$. Further suppose that $\int_{\Sigma_t} \Phi_{y(t),t_0} <\infty$ for all $t\in I$ with $t<t_0$. 

Then for all such times, we have that
\begin{equation}
\label{eq:ocmmcfthm}
\frac{d}{dt} \int_{\Sigma_t} \Phi_{y(t),t_0} = -\int_{\Sigma_t} \left| \vec{H} + \frac{(x-y-(t_0-t)y')^\perp}{2(t_0-t)} \right|^2 \Phi_{y,t_0} + \frac{1}{4}\int_{\Sigma_t}|(y')^\perp|^2 \Phi_{y,t_0} .
\end{equation}

In particular, the quantity $\exp\left(-\frac{1}{4}\int_t^{t_0} |y'(\tau)|^2 \d\tau\right) \int_{\Sigma_t} \Phi_{y(t),t_0}$ is non-increasing for $t<t_0$. 
\end{theorem}

Of course, if the centre does not change, that is, if $y(t)\equiv x_0$, then one recovers Huisken's monotonicity. The energy functional for curves in $\mathbb{R}^n$ is of course minimised by straight lines; in this case the equality case is easy to interpret:

\begin{corollary}
Fix $x_0,y_0\in\mathbb{R}^n$ and $t_0\in\mathbb{R}$ and let $\{\Sigma_t\}_{t\in I}$ be as above. 

Then for $t\in I$ with $t<t_0$, the quantity $\exp\left(-\frac{|y_0|^2}{4}(t_0-t)\right) \int_{\Sigma_t} \Phi_{x_0 + (t_0-t) y_0, t_0} $ is non-increasing, and is constant if and only if $\Sigma_t$ is a self-shrinking soliton which shrinks to $(x_0,t_0)$ and is orthogonal to $y$. 
\end{corollary}

\begin{proof}[Proof of Theorem \ref{thm:ocmmcf}]

For smooth, (spatially) compactly supported test functions $\phi=\phi(x,t)$ the first variation formula for mean curvature flow states that (see for instance \cite[Proposition 4.6]{ecker}) 
\begin{equation}
\label{eq:mcf1stvar}
\frac{d}{dt}\int_{\Sigma_t} \phi = \int_{\Sigma_t} \frac{\pr \phi}{\pr t} +\langle \vec{H}, D\phi\rangle - |\vec{H}|^2\phi.
\end{equation}

First we compute the gradient \begin{equation}\label{eq:mcfdphi}\frac{D\Phi_{y,t_0}}{\Phi_{y,t_0}}=-\frac{x-y}{2(t_0-t)},\end{equation}

Direct computation then yields that 
\begin{eqnarray}
\label{eq:mcfddt}
\frac{\pr \Phi_{y,t_0}}{\pr t} &+&\langle \vec{H}, D\Phi_{y,t_0}\rangle - |\vec{H}|^2\Phi_{y,t_0} \\&=&\nonumber   \Phi_{y,t_0} \left(\frac{k}{2(t_0-t)} -\frac{|x-y|^2}{4(t_0-t)^2}+\frac{\langle x-y,y'\rangle}{2(t_0-t)} -\frac{\langle \vec{H}, x-y\rangle}{2(t_0-t)} - |\vec{H}|^2\right).
\end{eqnarray}
By completing the square we then note that \begin{equation}\label{eq:mcfsq1} -\frac{|x-y|^2}{4(t_0-t)^2}+\frac{\langle x-y,y'\rangle}{2(t_0-t)} = -\frac{1}{4(t_0-t)^2}\left( |x-y-(t_0-t)y'|^2 - (t_0-t)^2|y'|^2\right).\end{equation} 

We will now make use of the divergence theorem (\ref{eq:mcfdivthm}). In particular, for fixed $t$ we set \begin{equation}\label{eq:mcfX} X = -\frac{1}{2(t_0-t)}(x-y-2(t_0-t)y').\end{equation} For this $X$ we compute that \begin{eqnarray}\nonumber \div_{\Sigma_t}(\Phi_{y,t_0} X)&=& \Phi_{y,t_0}\left( \div_{\Sigma_t} X -  \frac{\langle X, (x-y)^T\rangle}{2(t_0-t)} \right)\\&=&\label{eq:mcfdiv} \Phi_{y,t_0}\left(-\frac{k}{2(t_0-t)} + \frac{ |(x-y-(t_0-t)y')^T|^2 - (t_0-t)^2|(y')^T|^2  }{4(t_0-t)^2}\right),\end{eqnarray} where we have used that $x-y = (x-y-(t_0-t)y') + (t_0-t)y'$. 

Combining (\ref{eq:mcfddt} - \ref{eq:mcfdiv}), we thus have 
\begin{eqnarray}
\label{eq:mcfcomb}
\frac{\pr \Phi_{y,t_0}}{\pr t} &+&\langle \vec{H}, D\Phi_{y,t_0}\rangle - |\vec{H}|^2\Phi_{y,t_0} + \div_{\Sigma_t}(\Phi_{y,t_0}X) + \langle \vec{H},\Phi_{y,t_0}X\rangle \\&=&\nonumber     \Phi_{y(t),t_0} \left(  -\frac{|(x-y-(t_0-t)y')^\perp|^2}{4(t_0-t)^2} + \frac{|(y')^\perp|^2}{4} \right. \\&&\nonumber \qquad\qquad\left.- \frac{\langle \vec{H}, x-y-(t_0-t)y'\rangle}{(t_0-t)} - |\vec{H}|^2\right)\\&=&   \nonumber \Phi_{y(t),t_0} \left(\frac{|(y')^\perp|^2}{4} - \left| \vec{H} + \frac{(x-y-(t_0-t)y')^\perp}{2(t_0-t)}\right|^2  \right).\end{eqnarray} 

If the $\Sigma_t$ are compact, then we may immediately apply (\ref{eq:mcf1stvar}) with $\phi=\Phi_{y,t_0}$, and (\ref{eq:mcfdivthm}) to $\Phi_{y,t_0}X$ to conclude the result. If, however, the $\Sigma_t$ are noncompact, then for $R>0$ we select a smooth cutoff function $\chi=\chi_R$ on $\mathbb{R}^n$ such that $\chi_R=1$ on $B_R$, $\chi_R=0$ outside $B_{2R}$, with \begin{equation}\label{eq:mcfbound}R|D\chi_R| + R^2|D^2\chi_R| \leq C_0\end{equation} in between, where $C_0$ is a universal constant. 

Applying (\ref{eq:mcf1stvar}) with $\phi=\chi\Phi_{y,t_0}$ and (\ref{eq:mcfdivthm}) to $\chi \Phi_{y,t_0}X$, we then have
\begin{eqnarray}
\nonumber \frac{d}{dt}\int_{\Sigma_t} \chi\Phi_{y,t_0} &=& \int_{\Sigma_t} \chi\left( \frac{\pr \Phi_{y,t_0}}{\pr t} +\langle \vec{H}, D\Phi_{y,t_0}\rangle - |\vec{H}|^2\Phi_{y,t_0} + \div_{\Sigma_t}(\Phi_{y,t_0}X) + \langle \vec{H},\Phi_{y,t_0}X\rangle\right) \\&&  + \int_{\Sigma_t} \Phi_{y,t_0} \left( \frac{\pr \chi}{\pr t}+ \langle \vec{H},D\chi\rangle + \langle \nabla \chi, X\rangle\right).
\end{eqnarray}
Since $\chi$ is independent of time we of course have $\frac{\pr \chi}{\pr t}=0$. 

Now using the divergence theorem again to $ \Phi_{y,t_0} D\chi$, we have that 
\begin{equation}
 \int_{\Sigma_t} \Phi_{y,t_0} \left( \langle \vec{H},D\chi\rangle + \langle \nabla \chi, X\rangle\right) = \int_{\Sigma_t} \Phi_{y,t_0}\left( - \div_{\Sigma_t} D\chi + \langle \nabla \chi, X-\frac{D\Phi_{y,t_0}}{\Phi_{y,t_0}}\rangle\right).
 \end{equation}
 
We may simplify the last term using (\ref{eq:mcfX}) and (\ref{eq:mcfdphi}), so ultimately
\begin{eqnarray}
\label{eq:ocmmcffinal}
\nonumber \frac{d}{dt}\int_{\Sigma_t} \chi\Phi_{y,t_0} &=& \int_{\Sigma_t} \chi\Phi_{y,t_0}\left( \frac{|(y')^\perp|^2}{4} - \left| \vec{H} + \frac{(x-y-(t_0-t)y')^\perp}{2(t_0-t)}\right|^2 \right) \\&&  + \int_{\Sigma_t} \Phi_{y,t_0} \left(  -\div_{\Sigma_t} D\chi  +\langle \nabla \chi, y'\rangle \right).
\end{eqnarray}

By (\ref{eq:mcfbound}) we may estimate \begin{equation}|\div_{\Sigma_t} D\chi| \leq \frac{kC_0}{R^2}\mathbf{1}_{B_{2R}\setminus B_R},\end{equation}
\begin{equation}
|\langle \nabla \chi,y'\rangle| \leq |y'| \frac{C_0}{R} \mathbf{1}_{B_{2R}\setminus B_R}.
\end{equation}

Since by assumption $\int_{\Sigma_t} \Phi_{y,t_0} <\infty$, as in \cite[Theorem 4.13]{ecker} we may therefore let $R\rightarrow \infty$ to conclude the result. (For instance, one may use the bounds above to move the terms involving $y'$ and $D\chi$ to the left-hand side via an integrating factor, then apply the monotone convergence theorem to justify the remaining integral involving $\vec{H}$.)
\end{proof}

\begin{remark}
Theorem \ref{thm:ocmmcf} holds for Brakke flows as well, except that (\ref{eq:ocmmcfthm}) is instead an upper bound, since the first variation formula (\ref{eq:mcf1stvar}) is also an upper bound for such flows \cite{brakke}. 
\end{remark}

\subsection{Self-shrinkers}
\label{sec:shrinker}

In the study of self-shrinking solitons one may make the normalisation that the soliton flow $\Sigma_t$ shrinks to the origin at time 0. The flow is then determined by any negative time slice - in particular, $\Sigma_t = \sqrt{-t}\Sigma_{-1}$. The time slice $\Sigma = \Sigma_{-1}$ satisfies the elliptic equation $\vec{H} = -\frac{x^\perp}{2}$, and submanifolds satisfying this equation are called self-shrinkers. 

Colding-Minicozzi \cite{CMgeneric} introduced the entropy of a submanifold, defined by 
\begin{equation}
\lambda(\Sigma) = \sup_{x_0,t_0} F_{x_0,t_0}(\Sigma),
\end{equation}
where $F_{x_0,t_0}(\Sigma) = \int_\Sigma \rho_{x_0,t_0}$ is the Gaussian area as in (\ref{eq:gaussianarea}).

Using a moving-centre monotonicity formula, they were able to show that if $\Sigma$ is a self-shrinker, then its entropy is achieved at the natural centre $x_0=0$ and scale $t_0=1$, that is, $\lambda(\Sigma) = F_{0,1}(\Sigma)$. Their monotonicity is as follows:

\begin{proposition}[\cite{CMgeneric}, see also \cite{ketoverzhou}]
Let $\Sigma^k \subset \mathbb{R}^n$ be a self-shrinker and fix $a \in \mathbb{R}$. Then 
\begin{equation}
\frac{d}{ds} F_{sy, 1+as^2}(\Sigma) = -\frac{s}{2(1+as^2)^2}\int_\Sigma |(asx+y)^\perp|^2 \rho_{sy,1+as^2}
\leq  0,
\end{equation}
so long as $1+as^2>0$. In particular the Gaussian area $F_{x_0,t_0}(\Sigma)$ is maximised at $(x_0,t_0)=(0,1)$, and this maximum is strict unless $\Sigma$ is invariant under either dilation or a translation.
\end{proposition}

\section{$p$-Harmonic maps}
\label{sec:harm}

In this section we consider maps from a compact Riemmanian manifold $M^m$ (possibly with boundary) to a manifold $N$, which we assume to be isometrically embedded in $\mathbb{R}^n$. (Note that $n$ is not necessarily the dimension of $N$.) We fix $p \in (1,\infty)$. 

A ($W^{1,2}$) map $u:M\rightarrow N \hookrightarrow \mathbb{R}^n$ is said to be (weakly) $p$-harmonic if it is a (weak) solution of the elliptic equation \begin{equation} \Lap_p(u)=\div(|\nabla u|^{p-2}\nabla u) = -  A_u(\nabla u,\nabla u) ,\end{equation} where $A$ is the second fundamental form of $N$ in $\mathbb{R}^n$. Such maps are the critical points of the $L^p$ energy functional $\mathcal{E}_p(u) = \int_M |\nabla u|^p$. The case $p=2$ corresponds to the usual case of harmonic maps, whilst the limiting case $p\rightarrow 1$ corresponds to the case of minimal hypersurfaces - since if a level set $\{u=a\}$ is a smooth hypersurface, its mean curvature is given by $\div(\frac{\nabla u}{|\nabla u|})$ (see for instance \cite{evans}).

A $p$-harmonic map $u$ is said to be stationary if it additionally satisfies \begin{equation}\label{eq:statharm}\int_M |\nabla u|^p \div X = p \int_M  |\nabla u|^{p-2} \langle \nabla X, \nabla u\otimes \nabla u\rangle  \end{equation} for any smooth, compactly supported vector field $X$ on $M$. Explicitly, in coordinates, the expressions above are $|\nabla u|^2 = u^\alpha_{,i} u^\alpha_{,i}$ and $\langle \nabla X, \nabla u\otimes\nabla u\rangle = X_{i,j} u^\alpha_{,i} u^\alpha_{,j}$. 
For $1<p<2$ we understand products of $|\nabla u|^{p-2}$ with derivatives of $u$ to be zero on the critical set $\{\nabla u =0\}$. 

Note that $A_u(\nabla u,\nabla u)$ is orthogonal to $T_u N$. In particular, even for weak solutions (see \cite{struwe}, or \cite{schoenbook} for the case $p=2$) one has 
\begin{equation}
\label{eq:Aortho}
\int_M \langle A_u(\nabla u,\nabla u), \xi\rangle=0
\end{equation}
for any smooth map $\xi:M\rightarrow \mathbb{R}^n$ such that $\xi(x)\in T_{u(x)} N$ for all $x$. It follows that any smooth harmonic map is automatically stationary, for instance by applying the divergence theorem to the vector field on $M$ given by \begin{equation}\label{eq:ocmharmvfld}|\nabla u|^{p-2} \langle X\cdot \nabla u, \nabla u\rangle - \frac{1}{p}|\nabla u|^p X,\end{equation}  and noting that derivatives of $u$ must be tangent to $N$.

Monotonicity formulae play an important role in the regularity theory of $p$-harmonic maps, the most important case being when $M$ is a bounded domain in $\mathbb{R}^m$. We state the classical monotonicity formula for Euclidean balls $M=B_r^m\subset \mathbb{R}^m$. The statement below is found in \cite{naber}, but versions were proven earlier by other authors including: Schoen-Uhlenback \cite{schoenuhlenbeck} for minimising 2-harmonic maps, Price \cite{price} for stationary 2-harmonic maps and Hardt-Lin \cite{hardtlin} for minimising $p$-harmonic maps.

\begin{theorem}[\cite{naber}] 
Let $u:B_{\bar{r}}^m \rightarrow N$ be a stationary $p$-harmonic map, $1<p<m$. Then for $0<r<\bar{r}$, one has that 
\begin{equation}
\frac{d}{dr} \left( r^{p-m} \int_{B_r} |\nabla u|^p \right) = p r^{p-m} \int_{\pr B_r} |\nabla u|^{p-2}|\pr_r u|^2 = pr^{p-m-2} \int_{\pr B_r} |\nabla u|^{p-2} |x\cdot \nabla u|^2. 
\end{equation}

Equivalently, for $0<r<t<\bar{r}$, we have
\begin{equation}
t^{p-m}\int_{B_t} |\nabla u|^p - r^{p-m} \int_{B_r} |\nabla u|^p = p \int_{B_t\setminus B_r} |x|^{p-m-2} |\nabla u|^{p-2} |x\cdot \nabla u|^2.
\end{equation}

In particular, $r^{p-m} \int_{B_r} |\nabla u|^p$ is non-decreasing in $r$, and is constant if and only if $u$ is homogenous of degree zero. 
\end{theorem}

We use the technique of Theorem \ref{thm:ocmmin} (that is, morally, the $p=1$ case) to provide a family of moving-centre monotonicity formulae for stationary $p$-harmonic maps. For $p>1$ there is an excess term involving derivatives of $u$ in the $y$-direction, which can be handled by an explicit correction term, or by adjusting the family of balls. In the latter case, the motion of the centre must be constrained. Interpolating between these approaches gives rise to our family of monotonicity formulae.

\begin{definition}
Choose $q \in [1,p]$. Then for fixed $y\in B^m_{q^{-1/2}}$, we may define the nested family of balls \begin{equation} \label{eq:hmheatEs} E^{(q)}_s = B^m\left(sy, R_q(s)\right) \subset \mathbb{R}^m,\end{equation} 
where $R_q(s)=\sqrt{s(1-q|y|^2)+s^2q|y|^2}$. 

Note that for $q=1$, this family differs from Definition \ref{defn:Eballs} only by a rigid motion; instead of starting centred near $y$ and expanding through $B(0,1)$, here the balls $E^{(1)}_s$ will start centred near 0 and expand through $E^{(1)}_1 = B(y,1)$. For $q>1$ the $E^{(q)}_s$ foliate $\mathbb{R}^m$, whilst for $q=1$ they foliate the half-space defined by $\langle x,y\rangle > \frac{|y|^2-1}{2}$. 

As sub-level sets, we note that $E^{(q)}_s=\{0\leq f_q<s\}$ for \begin{equation}\label{eq:hmheatdeff} f_q(x)= \begin{cases} \frac{|x|^2}{1+2\langle x,y\rangle -|y|^2} &, q=1\\ \frac{-(1-q|y|^2 +2\langle x,y\rangle) + \sqrt{(1-q|y|^2 +2\langle x,y\rangle)^2 + 4(q-1)|x|^2|y|^2}}{2(q-1)|y|^2}&,q>1\end{cases}.\end{equation} 
\end{definition}

Our moving-centre monotonicity formulae for $p$-harmonic maps are then as follows:
\begin{theorem}
\label{thm:ocmharm1}
Let $p\in(1,m)$ and fix $q\in [1,p]$, $y\in B^m_{q^{-1/2}}$. Further, fix $\bar{s} \in (0,\infty)$ and let $u:E^{(q)}_{\bar{s}}\rightarrow N$ be a stationary $p$-harmonic map. Then for $0<s<\bar{s}$, one has that 
\begin{eqnarray}
\label{eq:ocmharmdiff}
\nonumber \frac{d}{ds}\left(s^{\frac{p-m}{2}}\int_{E_s^{(q)}} |\nabla u|^p \right) &=& \frac{qs^\frac{p-m+2}{2}}{2R_q(s)}\int_{\pr E_s^{(q)}} |\nabla u|^{p-2} (|y|^2 |\nabla u|^2 - |y\cdot \nabla u|^2) \\&&+  \frac{s^\frac{p-m-2}{2}}{2R_q(s)}\int_{\pr E_s^{(q)}} |\nabla u|^{p-2} \left({p|x\cdot \nabla u|^2-(p-q)s^2|y\cdot\nabla u|^2}\right).
\end{eqnarray}

Equivalently, for $0<s<t<\bar{s}$, we have 
\begin{eqnarray} 
\label{eq:ocmharmint}
\nonumber t^\frac{p-m}{2}\int_{E_t^{(q)}} |\nabla u|^p  &-& s^{\frac{p-m}{2}}\int_{E_s^{(q)}} |\nabla u|^p \\&=&q\int_{E_t^{(q)}\setminus E_s^{(q)}} |\nabla u|^{p-2} f^\frac{p-m+4}{2} \left(\frac{|y|^2|\nabla u|^2 - |y\cdot \nabla u|^2}{|x|^2+(q-1)f^2|y|^2}\right)  \nonumber \\&& + \int_{E_t^{(q)}\setminus E_s^{(q)}} |\nabla u|^{p-2} f^\frac{p-m}{2}\left(\frac{p|x\cdot \nabla u|^2-(p-q)f^2|y\cdot\nabla u|^2}{|x|^2+(q-1)f^2|y|^2}\right).  
\end{eqnarray}
\end{theorem}

Bounding the $y\cdot \nabla u$ terms using Cauchy-Schwarz yields the monotone quantities as follows:

\begin{corollary}
\label{cor:ocmharm}
Let $p\in(1,m)$ and fix $q\in (1,p]$, $y\in B^m_{q^{-1/2}}$. Let $u:E^{(q)}_{\bar{s}}\rightarrow N$ be a stationary $p$-harmonic map. Then for $0<s<\bar{s}$, the quantity 
\begin{equation}
s^{\frac{q-1}{p-1}\frac{p-m}{2}}\int_{E_s^{(q)}} |\nabla u|^p 
\end{equation}
is non-decreasing. 
\end{corollary}

In particular, when $q=p$, we get the sharper statement:

\begin{corollary}
\label{cor:ocmharm2}
Let $p\in(1,m)$ and fix $y\in B^m_{p^{-1/2}}$. Let $u:E^{(p)}_{\bar{s}}\rightarrow N$ be a stationary $p$-harmonic map. Then for $0<s<\bar{s}$, the quantity \begin{equation}
s^\frac{p-m}{2} \int_{E^{(p)}_s} |\nabla u|^p   
\end{equation}
is non-decreasing. If $y\neq 0$, then this quantity is constant if and only if $u$ is a constant map; if $y=0$ then it is constant if and only if $u$ is homogenous of degree zero. 
\end{corollary}

Our proof will use the coarea formula, and so we will also need the version of (\ref{eq:statharm}) for domains with boundary. Namely, let $\Omega \subset \mathbb{R}^m$ be a smooth bounded domain and let $X$ be a smooth vector field on $\Omega$. It follows from (\ref{eq:statharm}) that \begin{eqnarray}\label{eq:statharmbd} \int_\Omega |\nabla u|^{p-2}\left(|\nabla u|^2\div X - p \langle \nabla X, \nabla u\otimes \nabla u\rangle \right)\\\nonumber = \int_{\pr \Omega} |\nabla u|^{p-2}\left(|\nabla u|^2 X\cdot \nu - p\langle X\cdot \nabla u, \nu\cdot\nabla u\rangle\right)  ,\end{eqnarray} where $\nu$ is the outward unit normal to $\pr \Omega$. [One way to see this is to define $\varphi$ by $\varphi(t) = t/\epsilon$ for $0\leq t\leq \epsilon$, $\phi(t)=1$ for $t\geq \epsilon$, and let $\wt{\varphi}$ be a smooth approximation of $\varphi$. Apply (\ref{eq:statharm}) to $\wt{\varphi}(d(x,\pr\Omega)) X$ on $\Omega$ and let $\epsilon\rightarrow0$ using the coarea formula. For regular maps one may instead directly apply the divergence theorem with boundary to (\ref{eq:ocmharmvfld}).] 

\begin{proof}[Proof of Theorem \ref{thm:ocmharm1}]
For the proof we will suppress the dependence on $q$, that is, we fix $q$ and work with $E_s = E^{(q)}_s$, $f=f_q$, $R=R_q$. 

Since, by construction, the level sets of $f$ are the spheres $\pr E_s$ with centre and radius as in (\ref{eq:hmheatEs}), the function $f$ satisfies
\begin{equation}
\label{eq:centreradius}
|x-fy|^2 =R(f)^2=  f(1-q|y|^2) + f^2q|y|^2.
\end{equation}

It will be useful to record the expanded form \begin{equation}\label{eq:hmheatDs2}|x|^2 + (q-1)f^2|y|^2 = 2f\langle x,y\rangle+f(1-q|y|^2) + 2f^2(q-1)|y|^2.\end{equation} 

The outward unit normal $\nu$ of $\pr E_s$ with respect to $E_s$ is given by \begin{equation}\label{eq:harmnormal}\nu = \frac{x-sy}{|x-sy|} = \frac{x-sy}{R(s)}.\end{equation} But the gradient of $f$ on each $\pr E_s$ must be proportional to $\nu$, so implicitly differentiating using (\ref{eq:centreradius}) and then (\ref{eq:hmheatDs2}), we find that \begin{equation}\label{eq:hmheatDf}\frac{1}{2}\nabla f = \frac{x-fy}{1-q|y|^2 + 2\langle x,y\rangle+2(q-1)f |y|^2} = \frac{f(x-fy)}{|x|^2+(q-1)f^2|y|^2}.\end{equation} (One may also verify this directly using the explicit form of $f$.)

Taking the norm of both sides and using (\ref{eq:centreradius}) yields 
\begin{equation}
\label{eq:ocmharmnormdf}
|\nabla f| = \frac{2f R(f)}{|x|^2+(q-1)f^2|y|^2}. 
\end{equation}

The coarea formula gives that
\begin{equation}
\label{eq:harmcoarea}
\frac{d}{ds} \int_{E_s} |\nabla u|^p = \int_{\{f=s\}} \frac{|\nabla u|^p}{|\nabla f|} =\frac{1}{2sR(s)} \int_{\{f=s\}} |\nabla u|^p \left(|x|^2+(q-1)s^2|y|^2\right).
\end{equation}

But for fixed $s$, by the stationarity (\ref{eq:statharmbd}), for any vector field $X$ on $E_s$ which satisfies \begin{equation}\nabla_i X_j = \delta_{i,j},\end{equation} using the formula (\ref{eq:harmnormal}) for the normal we find that
\begin{equation}
(m-p)\int_{E_s} |\nabla u|^p =  \frac{1}{R(s)}\int_{\{f=s\}} |\nabla u|^{p-2}  \left( |\nabla u|^2 X\cdot (x-sy) - p\langle X\cdot \nabla u, (x-sy)\cdot\nabla u\rangle\right),
\end{equation}

Choosing $X=x+sy$ and polarising both terms on the right, we have 

\begin{eqnarray}
\nonumber (m-p)\int_{E_s} |\nabla u|^p &=& \frac{1}{R(s)}\int_{\{f=s\}} |\nabla u|^{p-2} \left( |\nabla u|^2 (|x|^2-s^2|y|^2) - p|x\cdot\nabla u|^2+ps^2|y\cdot\nabla u|^2 \right)
\\\nonumber&=&  \frac{1}{R(s)} \int_{\{f=s\}}|\nabla u|^p\left(|x|^2+(q-1)s^2|y|^2\right) \\&&\nonumber + \frac{qs^2}{R(s)}\int_{\{f=s\}}|\nabla u|^{p-2} \left(|y\cdot \nabla u|^2 - |y|^2|\nabla u|^2\right)  \\&& +\frac{1}{R(s)}\int_{\{f=s\}} |\nabla u|^{p-2} \left((p-q)s^2|y\cdot\nabla u|^2-p|x\cdot \nabla u|^2\right) .
\label{eq:ocmharmpol}
\end{eqnarray}

Using (\ref{eq:harmcoarea}) for the first term on the right and rearranging gives that

\begin{eqnarray}
\nonumber \frac{d}{ds}\left(s^{\frac{p-m}{2}}\int_{E_s} |\nabla u|^p \right) &=& \frac{qs^\frac{p-m+2}{2}}{2R(s)}\int_{\{f=s\}}|\nabla u|^{p-2} \left( |y|^2|\nabla u|^2-|y\cdot \nabla u|^2 \right)  \\ && +\frac{s^\frac{p-m-2}{2}}{2R(s)}\int_{\{f=s\}} |\nabla u|^{p-2} \left(p|x\cdot \nabla u|^2-(p-q)s^2|y\cdot\nabla u|^2\right).
\end{eqnarray}

This is the stated monotonicity formula in differential form (\ref{eq:ocmharmdiff}). On the other hand, using (\ref{eq:ocmharmnormdf}) and integrating using the coarea formula again will give the integral form (\ref{eq:ocmharmint}).

\end{proof}

\begin{proof}[Proof of Corollaries \ref{cor:ocmharm} and \ref{cor:ocmharm2}]
Bounding $|y\cdot \nabla u|^2\leq |y|^2 |\nabla u|^2$ in (\ref{eq:ocmharmpol}), and using (\ref{eq:harmcoarea}), we have that
\begin{eqnarray}
\nonumber (m-p) \int_{E_s}|\nabla u|^p &\leq& 2s \frac{d}{ds} \int_{E_s} |\nabla u|^p + \frac{p-q}{R(s)} \int_{\pr {E_s}} |\nabla u|^p s^2|y|^2 \\&\leq&2s \frac{d}{ds} \int_{E_s} |\nabla u|^p + \frac{p-q}{q-1} \frac{1}{R(s)} \int_{\pr {E_s}} |\nabla u|^p \left(|x|^2+(q-1)s^2|y|^2\right) \nonumber \\&=& 2s\, \frac{p-1}{q-1} \frac{d}{ds} \int_{E_s} |\nabla u|^p.
\end{eqnarray}
This implies the stated monotonicity. If $p=q$, then we did not lose anything in the second inequality above; therefore equality holds if and only if $x\cdot \nabla u\equiv 0$ and $|y\cdot\nabla u| = |y| |\nabla u|$, that is, if $y$ is parallel to $\nabla u^\alpha$ for each $\alpha$. The first condition implies that $u$ is homogenous of degree zero. If $y\neq 0$ then the second condition implies that $u$ is constant on lines orthogonal to $y$, which combined with the first forces $u$ to be constant. 
\end{proof}

\section{Harmonic map heat flow}
\label{sec:hmheat}

The harmonic map heat flow involves deforming a map $u:M^m \rightarrow N \hookrightarrow \mathbb{R}^n$ by the parabolic equation
\begin{equation}
\label{eq:hmheateq}
\pr_t u = \Lap u + A_u (\nabla u,\nabla u).
\end{equation}
Again $A$ is the second fundamental form of $N$ in $\mathbb{R}^n$. 

Struwe \cite{struwe} discovered a Gaussian-weighted monotonicity for regular solutions to (\ref{eq:hmheateq}) on $M=\mathbb{R}^m$. In this section we use $\Phi$ to denote the kernel \begin{equation}\Phi_{x_0,t_0}(x,t)= (4\pi (t_0-t))^{-\frac{m-2}{2}} \exp\left(- \frac{|x-x_0|^2}{4(t_0-t)}\right),\end{equation} and we denote $\mathbb{R}^m_t = \mathbb{R}^m \times \{t\} \subset \mathbb{R}^m \times I$. 

\begin{theorem}[\cite{struwe}]
Fix $x_0 \in \mathbb{R}^m$, $t_0\in I$. 
Let $u:\mathbb{R}^m\times I \rightarrow N$ be a regular solution of (\ref{eq:hmheateq}) with $|\nabla u| \leq c<\infty$ and $\int_{\mathbb{R}^m_t} |\nabla u|^2 \Phi_{x_0,t_0} <\infty$ for all $t\in I$ with $t<t_0$. Then 
\begin{equation}
\frac{d}{dt} \int_{\mathbb{R}^m_t} |\nabla u|^2 \Phi_{x_0,t_0} = -2 \int_{\mathbb{R}^m_t} \left| \pr_t u - \nabla u \cdot \frac{x-x_0}{2(t_0-t)} \right|^2 \Phi_{x_0,t_0} .
\end{equation}
In particular, $\int_{\mathbb{R}^m_t} |\nabla u|^2 \Phi_{x_0,t_0}$ is non-increasing for $t<t_0$. 
\end{theorem} 

Allowing the centre $x_0$ to move in time, we are able to prove the following moving-centre monotonicity for the harmonic map heat flow:

\begin{theorem}
\label{thm:ocmhmheat}
Fix $t_0\in I$ and a smooth curve $y=y(t)$ in $\mathbb{R}^m$. 
Let $u:\mathbb{R}^m\times I\rightarrow N$ be a regular solution of (\ref{eq:hmheateq}) with $|\nabla u| \leq c<\infty$ and $\int_{\mathbb{R}^m_t} |\nabla u|^2 \Phi_{y(t),t_0} <\infty$ for all $t$. Then 
\begin{equation}
\frac{d}{dt} \int_{\mathbb{R}^m_t} |\nabla u|^2 \Phi_{y(t),t_0} = -2 \int_{\mathbb{R}^m_t} \left| \pr_t u - \nabla u \cdot \frac{x-y - (t_0-t)y'}{2(t_0-t)} \right|^2 \Phi_{y,t_0} + \frac{1}{2} \int_{\mathbb{R}^m_t} |\nabla u \cdot y'|^2 \Phi_{y,t_0}.
\end{equation}
In particular,  $\exp\left(-\frac{1}{2} \int_t^{t_0} |y'(\tau)|^2 d\tau\right)\int_{\mathbb{R}^m_t} |\nabla u|^2 \Phi_{y(t),t_0}$ is non-increasing for $t<t_0$.
\end{theorem} 
\begin{proof}
The proof is similar to the proof of Theorem \ref{thm:ocmmcf}. By the exponential decay of the Gaussian weight and the assumed gradient bound, we may differentiate under the integral and integrate by parts freely: First we calculate

\begin{eqnarray}
\label{eq:hmheatddt}
\nonumber \frac{d}{dt} \int_{\mathbb{R}^m_t} |\nabla u|^2 \Phi_{y(t),t_0} &=& \int_{\mathbb{R}^m_t} 2\langle \nabla u,\nabla \pr_t u\rangle \Phi_{y,t_0} \\&& +\int_{\mathbb{R}^m_t}\Phi_{y,t_0} |\nabla u|^2\left( \frac{(x-y)\cdot y'}{2(t_0-t)} -\frac{|x-y|^2}{4(t_0-t)^2}  +\frac{m-2}{2(t_0-t)}\right).  
\end{eqnarray}

Integrating by parts, we have 
\begin{eqnarray}
\label{eq:hmheatibp}
\int_{\mathbb{R}^m_t} 2\langle \nabla u,\nabla \pr_t u\rangle \Phi_{y,t_0} &=& \int_{\mathbb{R}^m_t}\Phi_{y,t_0} \left( -2\langle \Lap u, \pr_t u\rangle + \frac{1}{t_0-t}\langle \pr_t u, (x-y)\cdot \nabla u\rangle\right)\\\nonumber&=& \int_{\mathbb{R}^m_t}\Phi_{y,t_0} \left( -2|\pr_t u|^2 + \frac{1}{t_0-t}\langle \pr_t u, (x-y)\cdot \nabla u\rangle\right),
\end{eqnarray}
where in the second line we used that $A_u(\nabla u,\nabla u)$ is orthogonal to $T_u N$ and hence to $\pr_t u $. 

By completing the square we have that \begin{equation}\label{eq:hmheatsq1} -\frac{|x-y|^2}{4(t_0-t)^2}+\frac{(x-y)\cdot y'}{2(t_0-t)} = -\frac{1}{4(t_0-t)^2}\left( |x-y-(t_0-t)y'|^2 - (t_0-t)^2|y'|^2\right).\end{equation} 

Since $u$ is regular, for a smooth compactly supported vector field $Y$ on $M$, applying the divergence theorem to the vector field $2\langle Y\cdot \nabla u, \nabla u\rangle -|\nabla u|^2 Y$ yields that 
\begin{eqnarray}
\label{eq:hmheatdiv}
 0&=& \int_{\mathbb{R}^m_t}  \left( 2\langle Y\cdot\nabla u,\Lap u\rangle +2 \langle \nabla Y, \nabla u\otimes \nabla u\rangle-|\nabla u|^2\div Y\right)
 \\\nonumber &=& \int_{\mathbb{R}^m_t}  \left( 2\langle Y\cdot\nabla u,\pr_t u\rangle + 2 \langle \nabla Y, \nabla u\otimes \nabla u\rangle -|\nabla u|^2\div Y\right),
\end{eqnarray}
where again we have used that $A_u(\nabla u,\nabla u)$ is orthogonal to $Y\cdot\nabla u$. 

For fixed $t$ we again set $X= -\frac{1}{2(t_0-t)}(x-y-2(t_0-t)y')$, so that by polarisation
\begin{equation}
\label{eq:hmheatdivX}
\div(\Phi_{y,t_0}X) = \Phi_{y,t_0}\left(-\frac{m}{2(t_0-t)} + \frac{|x-y-(t_0-t)y'|^2-(t_0-t)^2|y'|^2}{4(t_0-t)^2}\right),
\end{equation}  
 \begin{eqnarray}
 \nabla_i(\Phi_{y,t_0}X)_j &=& \Phi_{y,t_0}\left( X_{i,j} -  \frac{(x-y)_i X_j}{2(t_0-t)}\right) 
 \\\nonumber &=& \Phi_{y,t_0}\left( -\frac{\delta_{ij}}{2(t_0-t)} +  \frac{(x-y)_i (x-y-2(t_0-t)y')_j}{4(t_0-t)^2}\right). 
 \end{eqnarray}
 
  Contracting the last equation against the symmetric tensor $\nabla u \otimes \nabla u$ and polarising again then gives, for $Y=\Phi_{y,t_0}X$, 
 \begin{equation}
 \langle \nabla Y, \nabla u\otimes \nabla u\rangle = \Phi_{y,t_0} \left( -\frac{|\nabla u |^2}{2(t_0-t)} + \frac{|(x-y-(t_0-t)y')\cdot\nabla u|^2-(t_0-t)^2|y'\cdot\nabla u|^2}{4(t_0-t)^2} \right)  .
 \end{equation}
 
Again the exponential decay means that (\ref{eq:hmheatdiv}) in fact holds for $Y=\Phi_{y,t_0}X$, so subtracting the resulting identity from (\ref{eq:hmheatddt}) and using the calculations above we find that
\begin{eqnarray}
\nonumber \frac{d}{dt} \int_{\mathbb{R}^m_t} |\nabla u|^2 \Phi_{y(t),t_0} &=& \int_{\mathbb{R}^m_t}\Phi_{y,t_0} \left( -2|\pr_t u|^2 + \frac{1}{t_0-t}\langle \pr_t u, (x-y)\cdot \nabla u\rangle\right) \\&&+ \nonumber  \int_{\mathbb{R}^m_t}\Phi_{y,t_0}  \frac{1}{t_0-t}\langle\pr_t u,  (x-y-2(t_0-t)y')\cdot \nabla u\rangle\\\nonumber && -\int_{\mathbb{R}^m_t}\Phi_{y,t_0}\left( \frac{|(x-y-(t_0-t)y')\cdot\nabla u|^2-(t_0-t)^2|y'\cdot\nabla u|^2}{2(t_0-t)^2} \right)
\\ &=& -2 \int_{\mathbb{R}^m_t} \Phi_{y,t_0}\left|\pr_t u-\frac{x-y-(t_0-t)y'}{2(t_0-t)}\cdot\nabla u\right|^2 + \frac{1}{2}\int_{\mathbb{R}^m_t}\Phi_{y,t_0} |y'\cdot\nabla u|^2.
\end{eqnarray}
This concludes the proof.
\end{proof}

\bibliographystyle{plain}
\bibliography{ocmv8}
\end{document}